\documentclass[fleqn]{amsart}
\usepackage[all]{xy}
\usepackage{amsmath,amssymb,amsthm,mathrsfs,bm}

\theoremstyle{definition}
\newtheorem{dfn}{Definition}[section]

\newtheorem{rmk}[dfn]{Remark}

\theoremstyle{plain}
\newtheorem{prop}[dfn]{Proposition}
\newtheorem{lem}[dfn]{Lemma}
\newtheorem{thm}[dfn]{Theorem}

\def \etale{\'{e}tale }
\def \etalec{\'{e}tale, }
\def \motimes{\otimes_{\max}}
\def \C{\mathbb{C}}
\def \N{\mathbb{N}}

\title{On nuclearity of $C^*$-algebras of Fell bundles over \'{e}tale groupoids}
\author{Takuya Takeishi}
\address{Department of Mathematical Sciences, University of Tokyo}
\email{takeishi@ms.u-tokyo.ac.jp}
\date{}

\begin{document}

\begin{abstract}
In this paper, we show that if $E$ is a Fell bundle over an amenable \etale locally compact Hausdorff groupoid 
such that every fiber on the unit space is nuclear, then $C^*_r(E)$ is also nuclear. 
In order to show this result, we introduce (minimal) tensor products of Fell bundles with fixed $C^*$-algebras. 
\end{abstract}

\maketitle

\section{Introduction}
It is well-known that groupoid $C^*$-algebras can realize most of $C^*$-algebras related to dynamical systems
---group $C^*$-algebras, graph algebras, uniform Roe algebras and so on. 
In particular, $C^*$-algebras related to discrete dynamical systems can be represented by \etale groupoids.  
On the other hand, the crossed product associated with a group action on a (non-commutative) $C^*$-algebra 
is an important construction which is not included above. 
The Fell bundle over groupoids, originally defined for imprimitivity theorems, is a unified construction of groupoid $C^*$-algebras and 
crossed products. 
Hence, by Fell bundle $C^*$-algebras, we can construct most of $C^*$-algebras which we usually handle. 

There is a principle that the amenability of group-like objects corresponds to the nuclearity of $C^*$-algebras. 
For example, the following theorem is well-known (see e.g.~\cite{BO}): 
\begin{thm}\label{classical} 
The following hold: 
\begin{itemize}
\setlength{\itemsep}{-11pt}
\item[{\rm (i)}] A discrete group $\Gamma$ is amenable if and only if $C^*_r(\Gamma)$ is nuclear. \\
\item[{\rm (ii)}] An \etale locally compact Hausdorff groupoid $G$ is amenable if and only if $C^*_r(G)$ is nuclear. \\
\item[{\rm (iii)}] If an amenable discrete group $\Gamma$ acts on a nuclear $C^*$-algebra $A$, then the crossed product $A \rtimes_r \Gamma$ is nuclear. \\
\end{itemize}
\end{thm}
In this paper, we show that if $E$ is a Fell bundle over an amenable \etale locally compact Hausdorff groupoid 
such that every fiber on the unit space is nuclear, then $C^*_r(E)$ is also nuclear (Theorem \ref{nuclear}). 
Hence Theorem \ref{classical} is included in Theorem \ref{nuclear}, 
and Theorem \ref{nuclear} is the most general nuclearity theorem of $C^*$-algebras similar to discrete crossed products. 

In the case of discrete groups, Abadie-Vicens (\cite{A}) and Quigg (\cite[Corollary 2.17]{Q}) have already proved this theorem independently. 
Abadie-Vicens uses the method of tensor products of Fell bundles (over groups), 
while Quigg's proof relies on the correspondence between Fell bundles and coactions, and duality theorem for cocrossed products. 
The proof in this paper is relatively close to the one of Abadie-Vicens, but more simple and based on the method in Brown-Ozawa's book (\cite{BO}). 

\section{Preliminaries}

First, we fix some terminologies and notations. 
Let $G$ be a topological groupoid. 
An open set $S\subset G$ is called a {\it bisection} if the source and range maps  on $S$ are open homeomorphisms onto its ranges. 
A topological groupoid $G$ is said to be {\it \etale}if $G$ has an open base consisting of bisections. 
Throughout this paper, groupoids are always assumed to be \etalec locally compact and Hausdorff. 

For tensor products, the usual notation $\otimes$ denotes the minimal tensor product, 
and the algebraic tensor product is denoted by $\odot$. 

We use the terminology ``{\it Hilbert $C^*$-bimodules}" in the following sense: 

\begin{dfn}
Let $A,B$ be $C^*$-algebras. An $A$-$B$-bimodule $V$ is said to be a {\it Hilbert $A$-$B$-bimodule} if it satisfies the following:
\begin{itemize}
\setlength{\itemsep}{-11pt}
\item[{\rm (i)}] The module V is a left Hilbert $A$-module and a right Hilbert $B$-module. \\
\item[{\rm (ii)}] The equalities $\langle ax,y\rangle_r = \langle x,a^*y\rangle_r$ 
and $_l\langle x,yb \rangle = {}_l\langle xb^*,y\rangle$ hold for every $a \in A$, $b \in B$ and $x,y \in V$. \\
\item[{\rm (iii)}] We have $_l\langle x,y \rangle z = x\langle y,z \rangle_r$ for every $x,y,z \in V$. 
\end{itemize}
Here $_l\langle\ ,\ \rangle$ and $\langle\ ,\ \rangle_r$ denote the $A$-valued and $B$-valued inner products, respectively. 
\end{dfn}

Hilbert $C^*$-bimodules with the fullness condition are called imprimitivity biimodules. 
For elementary knowledge for imprimitivity bimodules, one can consult the book of Raeburn and Williams (\cite[Chapter 3]{RW}). 
Note that some propositions for imprimitivity bimodules in \cite{RW} are true for Hilbert $C^*$-bimodules (not all propositions, needless to say). 
For example, the norms of a Hilbert $C^*$-bimodule induced by the left and right inner products coincide (cf.~\cite[Proposition 3.11]{RW}). 
We tacitly use such propositions to Hilbert $C^*$-bimodules.

\subsection{Fell bundles}
If $X$ is a topological space and $p:E\rightarrow X$ is a bundle, the fiber on $x \in X$ is denoted by $E_x$. 
All Banach spaces are tacitly assumed to be complex Banach spaces. 
For the definition of Banach bundles, see the book of Fell and Doran (\cite[Chapter II, Definition 13.4]{FD1}). 

\begin{dfn}(cf.~\cite{K}) \label{fell}
Let $G$ be a groupoid and $p:E\rightarrow G$ be a Banach bundle. 
Set $E^{(2)} =\{ (e_1,e_2) \in E \times E \ |\ (p(e_1),p(e_2)) \in G^{(2)}\}$. 
Then $p:E\rightarrow G$ is a {\it Fell bundle} if there are continuous maps $E^{(2)} \rightarrow E,(e_1,e_2)\mapsto e_1e_2$ (called a {\it multiplication}) 
and $E\rightarrow E,e\mapsto e^*$(called an {\it involution}) which satisfy the following ten conditions:
\begin{itemize}
\setlength{\itemsep}{-11pt}
\item[{\rm (i)}] $p(e_1e_2)=p(e_1)p(e_2)$ for every $(e_1,e_2) \in E^{(2)}$. \\
\item[{\rm (ii)}] $p(e^*)=p(e)^{-1}$ for every $e \in E$. \\
\item[{\rm (iii)}] The multiplication $E_{\gamma_1}\times E_{\gamma_2} \rightarrow E_{\gamma_1 \gamma_2}$ is bilinear
 for every $(\gamma_1, \gamma_2) \in G^{(2)}$. \\
\item[{\rm (iv)}] The involution $E_{\gamma}\rightarrow E_{\gamma^{-1}}$ is conjugate-linear for every $\gamma \in G$. \\
\item[{\rm (v)}] $e_1(e_2e_3) = (e_1e_2)e_3$ for appropriate $e_1,e_2,e_3 \in E$. \\
\item[{\rm (vi)}] $e^{**}=e$ for every $e \in E$. \\
\item[{\rm (vii)}] $(e_1e_2)^*=e_2^*e_1^*$ for every $(e_1,e_2) \in E^{(2)}$. \\
\item[{\rm (viii)}] $\|e_1e_2\| \leq \|e_1\|\|e_2\|$ for every $(e_1,e_2) \in E^{(2)}$. \\
\item[{\rm (ix)}] $\|e^*e\|=\|e\|^2$ for every $e \in E$. \\
\item[{\rm (x)}] $e^*e\geq 0$ for every $e \in E$. 
\end{itemize}
\end{dfn}

Note that the fibers on the unit space $G^{(0)}$ are $C^*$-algebras by the axioms (i)--(ix), hence the axiom (x) makes sense. 
The restriction of $E$ on $G^{(0)}$ is denoted by $E^{(0)}$. 
From the above remark, $E^{(0)}$ is a $C^*$-algebraic bundle over $G^{(0)}$. 
The fiber $E_{\gamma}$ on $\gamma \in G$ is a Hilbert $E_{r(\gamma)}$-$E_{s(\gamma)}$-bimodule. 
Here and henceforth, we omit the projection $p$ and simply say that $E$ is a Fell bundle. 

Let $E$ be a Fell bundle over a groupoid $G$. We denote by $C_c(E)$ the space of compactly-supported sections, 
and define a multiplication and an involution on $C_c(E)$ by
\begin{equation*}
f*g(\gamma) = \sum_{\gamma = \alpha \beta} f(\alpha)g(\beta),\ f^*(\gamma) = f(\gamma^{-1})^*
\end{equation*}
for $f,g \in C_c(E)$. 
Then there is a unique $C^*$-completion $C^*_r(E)$ of $C_c(E)$ such that the natural restriction map $C_c(E)\rightarrow C_c(E^{(0)})$ 
extends to a {\it faithful} conditional expectation $C^*_r(E)\rightarrow C_0(E^{(0)})$. 
The $C^*$-algebra $C^*_r(E)$ is called the reduced $C^*$-algebra of $E$.

By definition, Fell bundles over (locally compact) groups are original Fell bundles (\cite{FD2}), 
and Fell bundles over locally compact spaces are $C^*$-algebraic bundles. 

See \cite{K} for the detail of basics of Fell bundles. 

\subsection{Construction of Bundles}
In this section, we introduce a basic technique to construct bundles. 
When one constructs a Fell bundle $E$ over a groupoid $G$ from the given fibers $\{E_{\gamma}\}_{\gamma \in G}$, 
one can specify continuous sections of $E$ instead of topologizing $E= \coprod_{\gamma} E_{\gamma}$ directly. 
The proof of the following proposition is based on \cite[Chapter VIII, 2.4]{FD2}.  

\begin{prop} \label{construction}
Let $G$ be a groupoid. 
Let $E = \coprod_{\gamma} E_{\gamma}$ be an untopologized Fell bundle over $G$, 
i.e., every $E_{\gamma}$ is a Banach space and $E$ has a multiplication and an involution satisfying 
{\rm (i)--(x)} in {\rm Definition \ref{fell}}. 
Let $A_{0}$ be a $*$-algebra of compactly-supported sections of $E$. 
Assume the following:
\begin{itemize}
\item[{\rm (i)}] The subset $\{ f(\gamma)\ |\ f \in A_0 \}$ is dense in $E_{\gamma}$ for any $\gamma \in G$. 
\item[{\rm (ii)}] The function $\gamma \in G \longmapsto \| f(\gamma) \|$ is continuous for any $f \in A_0$. 
\end{itemize}
Then, there exists a unique topology of $E$ which makes $E$ a Fell bundle and all sections in $A_0$ continuous. 
\end{prop}
\begin{proof}
By \cite[Chapter II, Theorem 13.18]{FD1}, $E$ has a unique Banach bundle structure which makes all sections in $A_0$ continuous. 
We have to show the continuity of the multiplication and of the involution of $E$. 
We will show the continuity only for the multiplication and omit the proof for the involution, since the proof is almost the same. 
Let $\{a_i\}$ and $\{b_i\}$ be two nets in $E$ such that $(a_i, b_i) \in E^{(2)}$ and $(a_i, b_i)$ converges to $(a, b) \in E^{(2)}$. 
It suffices to show that $a_ib_i$ converges to $ab$. 

If $a=0$, then we have $\| a_ib_i \| \leq \| a_i \| \| b_i \| \rightarrow 0$ and hence $a_ib_i$ converges to $0=ab$, 
and the same holds if $b=0$. 
Thus we may assume $a,b \neq 0$.  
Let $\alpha_i = p(a_i)$, $\alpha = p(a)$, $\beta_i = p(b_i)$, and $\beta = p(b)$, where $p:E\rightarrow G$ is the canonical projection. 
Then $(\alpha_i, \beta_i),\ (\alpha,\beta) \in G^{(2)}$, and $(\alpha_i, \beta_i)$ converges to $(\alpha, \beta)$. 
Let $\varepsilon > 0$. 
By the assumption (i), there exist $f, g \in A_0$ such that 
\begin{eqnarray*}
\| f(\alpha) - a \| < \frac{\varepsilon }{\|b\|},\ \mbox{and}\  
\| g(\beta) - b \| < \frac{\varepsilon }{\|f(\alpha)\|}. 
\end{eqnarray*}
Since $A_0$ is closed under the convolution product, $f*g$ is a continuous section. 
Let $\alpha \in S_1 \subset T_1$ and $\beta \in S_2 \subset T_2$ be bisections of open neighbourhoods of $\alpha$ and $\beta$. 
Set $S=S_1S_2$ and $T=T_1T_2$. 
Take a continuous function $\varphi:G\rightarrow \C$ such that $\varphi \equiv 1$ on $S$ and $\varphi \equiv 0$ out of $T$. 
Then, $h(\gamma) = \varphi(\gamma) f*g(\gamma)$ is a continuous section of $E$ by the continuity of the scalar multiplication. 
We have $h(\alpha \beta) = f(\alpha)g(\beta)$, and $h(\alpha_i \beta_i)=f(\alpha_i)g(\beta_i)$ for large $i$. 

Since $f$ and $g$ are continuous, we have
\begin{eqnarray*}
\| f(\alpha_i) - a_i \| < \frac{\varepsilon }{\|b_i\|}\ \mbox{and}\  
\| g(\beta_i) - b_i \| < \frac{\varepsilon }{\|f(\alpha_i)\|} 
\end{eqnarray*}
for large $i$. 
Then, 
\begin{eqnarray*}
\| a_ib_i - h(\alpha_i \beta_i) \| = \| a_ib_i - f(\alpha_i) g(\beta_i) \| \leq 2\varepsilon,
\end{eqnarray*}
and also we have
\begin{eqnarray*}
\| ab - h(\alpha \beta) \| \leq 2\varepsilon .
\end{eqnarray*}
Hence, $a_ib_i$ converges to $ab$ by \cite[Chapter II, Proposition 13.12]{FD1}. 
\end{proof}

\section{Tensor Products of Fell Bundles with $C^*$-algebras}

\subsection{The case of locally compact spaces}
In this section, we summarize some known results of the tensor product of a $C^*$-algebraic bundle over a locally compact space with a fixed $C^*$-algebra. 

\begin{dfn}
Let $E$ be a $C^*$-algebraic bundle over a locally compact Hausdorff space $X$, and let $A$ be a $C^*$-algebra. 
Then $E \otimes A$ is said to be a {\it continuous bundle} if 
\begin{equation*}
x \in X \longmapsto \left\| \sum_{i=1}^n f_i(x)\otimes a_i \right\|_{\min}
\end{equation*}
is continuous for every $f_1,\cdots,f_n \in C_c(E)$ and $a_1,\cdots, a_n \in A$. 
\end{dfn}

If $E \otimes A$ is a continuous bundle, then one can define a $C^*$-algebraic bundle 
$E \otimes A = \coprod_{x \in X} E_x \otimes A$ over $X$ by specifying that every element of $C_c(E) \odot A$ defines a continuous section. 
Unfortunately, $E \otimes A$ is not always continuous. 
For example, if $A$ is a {\it non-exact} $C^*$-algebra, then one can construct a $C^*$-algebraic bundle $E$ over $\hat{\N}$ 
such that $E \otimes A$ is not continuous, where $\hat{\N}$ is the one-point compactification of $\N$ (cf.~\cite[Section 4]{KW}). 
The continuity of the tensor product is a hard problem and studied by Kirchberg and Wassermann in \cite{KW}. 
However, such a problem does not occur when every fiber is a nuclear $C^*$-algebra. 

\begin{prop}{\rm (Kirchberg-Wassermann \cite{KW})}\label{k-w} 
Let $E$ be a $C^*$-algebraic bundle over a locally compact Hausdorff space $X$ such that $E_x$ is nuclear for every $x \in X$. 
Then, $E \otimes A$ is a continuous bundle for any $C^*$-algebra $A$. 
\end{prop}
\begin{proof}
See \cite[Section 2]{KW}.
\end{proof}

See also \cite{KW} for the detail of the proof of the following propositions: 

\begin{prop} \label{sec}
If $E\otimes A$ is a continuous bundle, then $C_0(E\otimes A) = C_0(E) \otimes A$. 
\end{prop}

\begin{prop} \label{k-w2}
If $E_x$ is nuclear for every $x \in X$, then $C_0(E)$ is nuclear. 
\end{prop}
\begin{proof}
If $A$ is a $C^*$-algebra, then we have
\begin{equation*}
C_0(E)\otimes_{\max}A = C_0(E\otimes_{\max}A)=C_0(E\otimes A)=C_0(E)\otimes A
\end{equation*}
since every fiber of $E$ is nuclear. 
Hence $C_0(E)$ is nuclear. 
\end{proof}

\subsection{The general case}

Let $E$ be a Fell bundle over a groupoid $G$, and $A$ be a $C^*$-algebra. 
Assume that $E^{(0)}\otimes A$ is a continuous bundle.  
We will define the tensor product $E\otimes A$. 

Let $\gamma \in G$ and $x=s(\gamma), y=r(\gamma)$. 
Then, $E_{\gamma}$ is a Hilbert $E_y$-$E_x$-bimodule. 
We regard $A$ as an imprimitive $A$-$A$-bimodule in the obvious way. 
Let $(E\otimes A)_{\gamma} = E_{\gamma}\otimes A$, where $E_{\gamma}\otimes A$ is the exterior tensor product as Hilbert $C^*$-bimodules (cf.~\cite[Chapter 3]{RW}). 
Then, $(E\otimes A)_{\gamma}$ is a Hilbert $(E_y\otimes A)$-$(E_x\otimes A)$-bimodule endowed with the inner products 
\begin{eqnarray*}
&&\langle e_1\otimes a_1, e_2\otimes a_2\rangle_r = e_1^*e_2 \otimes a_1^*a_2,\\
&&_l\langle e_1\otimes a_1, e_2\otimes a_2\rangle = e_1e_2^* \otimes a_1a_2^*, 
\end{eqnarray*}
for $e_1,e_2 \in E_{\gamma}$ and $a_1,a_2 \in A$. 

Next, we define a multiplication and an involution on $E \otimes A = \coprod_{\gamma \in G} (E\otimes A)_\gamma$. 
Let $E \odot A = \coprod_{\gamma \in G} (E\odot A)_\gamma$ and define a multiplication and an involution on $E \odot A$ by 
\begin{eqnarray*}
&&(e_1\otimes a_1)(e_2\otimes a_2) = e_1e_2 \otimes a_1a_2,\\
&&(e \otimes a)^* = e^*\otimes a^* 
\end{eqnarray*}
for $(e_1,e_2) \in E^{(2)}, e \in E$ and $a_1,a_2,a \in A$. 
For $s,t \in (E\odot A)_{\gamma}$, we have 
\begin{eqnarray*}
\langle s,t\rangle_r = s^*t\ \mbox{and}\ _l\langle s,t\rangle = st^*. 
\end{eqnarray*} 

\begin{lem}{\rm (cf.~\cite[Proposition 3.8]{A})}\label{ineq} 
If $s,t \in E\odot A$ are composable, then $\| st\| \leq \|s\| \|t\|$. 
\end{lem}
\begin{proof}
We have 
$\| st \|^2 = \|t^*s^*st\| = \| \langle t, (s^*s)t \rangle_r \| \leq \|t\| \|(s^*s)t\| \leq \|t\|^2 \|s\|^2$.
\end{proof}

By Lemma \ref{ineq}, the multiplication map $(E_{\alpha}\odot A) \times (E_{\beta}\odot A) \rightarrow E_{\alpha \beta}\odot A$ 
for $(\alpha, \beta) \in G^{(2)}$ is extended continuously to $(E_{\alpha}\otimes A) \times (E_{\beta}\otimes A) \rightarrow E_{\alpha \beta}\otimes A$. 
Therefore, the multiplication is extended to whole $E\otimes A$. 
Similarly, the involution is extended to whole $E\otimes A$ since the involution is isometric. 
By the construction, $E\otimes A$ is an untopologized Fell bundle over $G$ in Proposition \ref{construction}. 

\begin{lem} \label{conti}
For $f_1,\cdots f_n \in C_c(E)$ and $a_1,\cdots a_n \in A$, 
\begin{equation*}
\gamma \in G \longmapsto \left\| \sum_{i=1}^n f_i(\gamma)\otimes a_i \right\|
\end{equation*}
is a continuous function. 
\end{lem}
\begin{proof}
Let $f_1,\cdots f_n \in C_c(E)$ and $a_1,\cdots a_n \in A$. 
It suffices to show that the above function is continuous on every bisection of $G$. 
Let $S \subset G$ be a bisection. 
Then for $f,g \in C_c(E)$,
\begin{equation*}
x \in s(S) \longmapsto f(Sx)^*g(Sx) \in E_x
\end{equation*}
is a continuous section of $E^{(0)}$. 
Therefore, 
\begin{equation*}
x \in s(S) \longmapsto \left\| \sum_{i,j=1}^n f_i(Sx)^*f_j(Sx)\otimes a_i^*a_j \right\|
\end{equation*}
is continuous since $E^{(0)}\otimes A$ is a continuous bundle by the assumption. 
Now it is easy to see that 
\begin{equation*}
\gamma \in S \longmapsto \left\| \sum_{i=1}^n f_i(\gamma)\otimes a_i \right\| = \left\| \sum_{i,j=1}^n f_i(\gamma)^*f_j(\gamma)\otimes a_i^*a_j \right\|^{1/2}
\end{equation*}
is continuous. 
\end{proof}

\begin{prop}
The bundle $E\otimes A$ has a unique Fell bundle structure such that
\begin{equation*}
\gamma \longmapsto \sum_{i=1}^n f_i(\gamma) \otimes a_i \in (E\otimes A)_{\gamma}
\end{equation*}
is a continuous section for every $\sum_{i=1}^n f_i \otimes a_i \in C_c(E) \odot A$.
\end{prop}
\begin{proof}
It can be easily seen that the usual product of $C_c(E)\odot A$ coincides with the convolution product as sections, 
and the usual involution with the one as sections. 
Therefore, by using Lemma \ref{conti}, we can apply Proposition \ref{construction} to the untopologized Fell bundle $E\otimes A$. 
\end{proof}

Next, we investigate the reduced $C^*$-algebra of $E\otimes A$. 
From the above construction, we have $(E\otimes A)^{(0)} = E^{(0)}\otimes A$. 
Since $C_0(E^{(0)}\otimes A) = C_0(E^{(0)}) \otimes A$ by Proposition \ref{sec}, 
$C_0(E^{(0)}) \otimes A$ is a $C^*$-subalgebra of $C^*_r(E\otimes A)$ with a canonical faithful conditional expectation
$\tilde{P}:C^*_r(E\otimes A) \rightarrow C_0(E^{(0)}) \otimes A$. 

We use an elementary lemma about conditional expectations. 

\begin{lem} \label{expec}
Let $A,B$ be $C^*$-algebras and $C$ be a common $C^*$-subalgebra of $A$ and $B$. 
Let $A_0 \subset A$, $B_0 \subset B$ be dense $*$-subalgebras and $\pi:A_0\rightarrow B_0$ be a surjective $*$-homomorphism. 
Assume that there exist a conditional expectation $P_A:A\rightarrow C$ and a faithful conditional expectation $P_B:B\rightarrow C$ such that the diagram 
\begin{eqnarray*}
\xymatrix{
A_0 \ar[dr]_{P_A} \ar[rr]^{\pi} & & B_0 \ar[dl]^{P_B} \\
 & C & }
\end{eqnarray*}
commutes. Then $\pi$ extends to a surjective $*$-homomorphism $\tilde{\pi}:A\rightarrow B$. 

Moreover, if $P_A$ is faithful, then $\tilde{\pi}$ is injective. 
\end{lem}
\begin{proof}
Since $P_B$ is faithful, we have
\begin{equation*}
\|b\|^2 = \sup \left\{ \|P_B(c^*b^*bc)\| \ : \ c \in B_0, P_B(c^*c) \leq 1 \right\} 
\end{equation*}
for every $b \in B$. Hence, for $a \in A_0$, 
\begin{eqnarray*}
\| \pi(a) \|^2 &=& \sup \left\{ \|P_B(\pi(c)^*\pi(a)^*\pi(a)\pi(c))\| \ : \ c \in A_0, P_B(\pi(c)^*\pi(c)) \leq 1 \right\} \\
&=& \sup \left\{ \|P_A(c^*a^*ac)\| \ : \ c \in A_0, P_A(c^*c) \leq 1 \right\} \\
&\leq& \|a\|^2.
\end{eqnarray*}
Therefore, $\pi$ extends to a surjection $\tilde{\pi}:A\rightarrow B$. 
If $P_A$ is also faithful, then the equality holds in the last inequality. Hence $\tilde{\pi}$ is isometric in this case. 
\end{proof}

\begin{prop} \label{min}
The $C^*$-algebra $C^*_r(E \otimes A)$ is canonically isomorphic to $C^*_r(E) \otimes A$. 
\end{prop}
\begin{proof}
Let $P:C^*_r(E) \rightarrow C_0(E^{(0)})$ be the canonical faithful conditional expectation. 
Then $P\otimes \mathrm{id} : C^*_r(E) \otimes A \rightarrow C_0(E^{(0)}) \otimes A$ is a faithful conditional expectation (cf.~\cite[Appendix]{Av}). 
One can apply Lemma \ref{expec} to the diagram
\begin{eqnarray*}
\xymatrix{
C_c(E) \odot A \ar[dr]_{\tilde{P}} \ar[rr]^{\mathrm{id}} & & C_c(E) \odot A \ar[dl]^{P\otimes \mathrm{id}} \\
 & C_0(E^{(0)}) \otimes A &. }
\end{eqnarray*}
Hence the extension of the identity map gives the isomorphism of $C^*_r(E \otimes A)$ and $C^*_r(E) \otimes A$. 
\end{proof}

\section{The Main Theorem}
\begin{thm} \label{nuclear}
Let E be a Fell bundle over an \etale locally compact Hausdorff groupoid G. 
If G is amenable, then the following conditions are equivalent:
\begin{itemize}
\setlength{\itemsep}{-11pt}
\item[{\rm (i)}] The $C^*$-algebra $C^*_r(E)$ is nuclear.\\
\item[{\rm (ii)}] The fiber $E_x$ is nuclear for every $x \in G^{(0)}$.\\
\item[{\rm (iii)}] The $C^*$-algebra $C_0(E^{(0)})$ is nuclear.
\end{itemize}
\end{thm}

Since there exists a canonical faithful conditional expectation from $C^*_r(E)$ onto $C_0(E^{(0)})$, (i) implies (iii). 
Since $E_x$ is a quotient of $C_0(E^{(0)})$, (iii) implies (ii). 
The condition (ii) implies (iii) by Proposition \ref{k-w2}. 
Therefore, the implication (ii) $\Rightarrow $ (i) is the only non-trivial part of this theorem. 

The proof we give here is based on the proof of \cite[Theorem 5.6.18]{BO}. 
The first lemma is the general property concerning the relation between positive definite functions and 
contractive completely positive (c.c.p.) maps. 
We think this lemma is useful for other purposes. 

\begin{lem} \label{ccp}
Let $E$ be a Fell bundle over a groupoid $G$, 
and let $h$ be a compactly supported continuous positive definite function on $G$ with $\sup_{\gamma \in G} |h(\gamma)| \leq 1$. 
Then the multiplier map 
\begin{equation*}
m_h : C_c(E) \rightarrow C_c(E),\ f \mapsto hf
\end{equation*}
extends to a c.c.p.~map on $C^*_r(E)$. 
\end{lem}
\begin{proof}
For $x \in G^{(0)}$, let $V_x = \bigoplus_{\gamma \in G_x} E_{\gamma}$ (the direct sum is taken as right Hilbert $E_x$-modules) 
and define a representation $\pi_x: C^*_r(E) \rightarrow B(V_x)$ by 
\begin{equation*}
\pi_x(f) \left( \sum_{\gamma \in G_x} \xi_{\gamma} \right) = \sum_{\gamma \in G_x} \left( \sum_{\beta \in G_x} f(\gamma \beta^{-1})\xi_{\beta} \right)
\end{equation*}
for $f \in C_c(E)$ and $\sum_{\gamma} \xi_{\gamma} \in V_x$. 
Then $\{\pi_x\}_{x \in G^{(0)}}$ is a faithful family of representations of $C^*_r(E)$ (cf.~\cite{K}). 
Similarly, $\lambda_x: C^*_r(G) \rightarrow B(\ell^2(G_x))$ is defined by 
\begin{equation*}
\lambda_x(f)\xi(\gamma) = \sum_{\beta \in G_x} f(\gamma \beta^{-1})\xi(\beta)
\end{equation*} 
for $f \in C_c(G)$ and $\xi \in \ell^2(G_x)$. 

Define $\iota : V_x \rightarrow \ell^2(G_x) \otimes V_x$ by  
\begin{equation*} 
\sum_{\gamma \in G_x}\xi_{\gamma} \mapsto \sum_{\gamma \in G_x} \delta_{\gamma} \otimes \xi_{\gamma}. 
\end{equation*} 
We can see that $\iota$ is adjointable and isometric. The adjoint map $\iota^*$ is given by 
\begin{equation*}
\iota^* \left( \sum_{\alpha \in G_x}  \delta_{\alpha} \otimes \left( \sum_{\beta \in G_x} \xi_{\alpha,\beta} \right) \right) 
= \sum_{\alpha \in G_x} \xi_{\alpha,\alpha} 
\end{equation*}
for $\xi_{\alpha,\beta} \in E_{\beta}$. 
Define $T_x:V_x \rightarrow \ell^2(G_x) \otimes V_x$ by 
\begin{equation*}
T_x\left( \sum_{\gamma \in G_x}\xi_{\gamma} \right) = \sum_{\gamma \in G_x} \lambda_x(h)^{1/2} \delta_{\gamma} \otimes \xi_{\gamma}. 
\end{equation*}
Then $T_x$ is adjointable and contractive since $T_x = (\lambda_x(h)^{1/2}\otimes 1) \circ \iota$. 
By simple calculation, we have $T_x^*(1\otimes \pi_x(f))T_x = \pi_x(m_h(f))$ for every $f \in C_c(E)$. 
Therefore, if one identifies $C^*_r(E)$ with $(\bigoplus_x \pi_x)(C^*_r(E))$, the restriction of the c.c.p.~map  
\begin{equation*}
\Phi : \prod_x B(V_x) \rightarrow \prod_x B(V_x),\ \sum_x a_x \mapsto \sum_x T_x^*(1\otimes a_x)T_x
\end{equation*}
on $C^*_r(E)$ gives the extension of $m_h$. 
\end{proof}

In fact, the condition that $h$ is compactly supported is not necessary, but we need only the case of compactly-supported. 

Let $E$ be a Fell bundle over an amenable groupoid $G$ such that $E_x$ is nuclear for every $x \in G^{(0)}$, and let $A$ be an arbitrary $C^*$-algebra. 
Then $E^{(0)} \otimes A$ is continuous by Proposition \ref{k-w}. 

\begin{lem} \label{bdd}
Let $K$ be a compact subset of $G$. Then the following hold: 
\begin{itemize}
\setlength{\itemsep}{-2pt}
\item[{\rm (i)}] 
There exists a positive constant $C_K>0$ such that for $f_1,\cdots,f_n \in C_c(E)$ supported in $K$ 
and $a_1,\cdots,a_n \in A$, we have 
\begin{equation*}
\left\| \sum_{i=1}^n f_i\otimes a_i \right\|_{\max} \leq C_K \sup_{\gamma \in K} \left\| \sum_{i=1}^n f_i(\gamma) \otimes a_i \right\|_{E_{\gamma}\otimes A}.
\end{equation*}
\item[{\rm (ii)}] 
For $f_1,\cdots,f_n \in C_c(E)$ and $a_1,\cdots,a_n \in A$, we have
\begin{equation*}
\sup_{\gamma \in K} \left\| \sum_{i=1}^n f_i(\gamma) \otimes a_i \right\|_{E_{\gamma}\otimes A} \leq \left\| \sum_{i=1}^n f_i\otimes a_i \right\|_{\min}.
\end{equation*}
\end{itemize}
Here, $\| \cdot \|_{\max}$ and $\| \cdot \|_{\min}$ are the maximal and minimal norms of $C^*_r(E)\odot A$. 
\end{lem}
\begin{proof}
The statement (ii) follows from Proposition \ref{min}. We prove (i). 
By the usual argument of the partition of unity, we may assume that $K$ is contained in some bisection. 
Then $f_i^*f_j$ is in $C_0(E^{(0)})$ for every $i,j$. 
Since $C_0(E^{(0)})$ is nuclear, the restriction of the maximal tensor norm to $C_0(E^{(0)}) \odot A$ coincides with the minimal tensor norm, and
$C_0(E^{(0)})\otimes A = C_0(E^{(0)}\otimes A)$. 
Therefore, we have
\begin{eqnarray*}
\left\| \sum_{i=1}^n f_i\otimes a_i \right\|_{\max}^2 
= \left\| \sum_{i,j} f_i^*f_j\otimes a_i^*a_j \right\|_{\max} 
= \sup_{x \in G^{(0)}} \left\| \sum_{i,j} f_i^*f_j(x)\otimes a_i^*a_j \right\| \\
= \sup_{\gamma \in G} \left\| \sum_{i,j} f_i(\gamma)^*f_j(\gamma) \otimes a_i^*a_j \right\| 
= \sup_{\gamma \in G} \left\| \sum_i f_i(\gamma) \otimes a_i \right\|_{E_{\gamma}\otimes A}^2. 
\end{eqnarray*}
\end{proof}

These lemmata are all we need, and now we give the proof of Theorem \ref{nuclear}. 
In order to prove the theorem, it suffices to show that the quotient map $Q:C^*_r(E) \motimes A \rightarrow C^*_r(E)\otimes A$ is injective. 

For every compact subset $K$ of $G$, We denote by $C(K,E)$ the set of continuous sections of $E$ supported in $K$. 
Then $C(K,E)$ is complete with respect to the sup-norm 
and $C(K,E) \odot A$ is a dense subspace of $C(K,E\otimes A)$ with the sup-norm topology. 

Let $h$ be a continuous positive definite function on $G$ supported in a compact subset $K$ of $G$, 
and assume $\sup_{\gamma \in G} |h(\gamma)| \leq 1$. 
Then by Lemma \ref{bdd}, the inclusion map $C(K,E) \odot A \rightarrow C^*_r(E) \odot A$ extends to bounded injective maps 
$C(K, E\otimes A) \rightarrow C^*_r(E)\motimes A$ and $C(K, E\otimes A) \rightarrow C^*_r(E)\otimes A$ with closed images. 
Thus we can regard $C(K, E\otimes A)$ as a closed subspace of $C^*_r(E)\motimes A$ and $C^*_r(E)\otimes A$. 
Then the quotient map $Q$ is injective on $C(K, E\otimes A)$. 

By Lemma \ref{ccp}, we have the following commutative diagram: 
\begin{eqnarray*}
\xymatrix{
C^*_r(E) \motimes A \ar[dd]_{Q} \ar[rr]^{m_h \motimes \mathrm{id}} & & C^*_r(E) \motimes A \ar[dd]^{Q} & \\
\\
C^*_r(E) \otimes A \ar[rr]^{m_h \otimes \mathrm{id}} & & C^*_r(E) \otimes A & .}
\end{eqnarray*}
Let $a \in C^*_r(E)\motimes A$ with $Q(a)=0$. Then we have $Q\circ (m_h \motimes \mathrm{id})(a) = (m_h \otimes \mathrm{id})\circ Q(a) = 0$. 
Since the image of $m_h \motimes \mathrm{id}$ is contained in $C(K, E\otimes A)$ and $Q$ is injective on $C(K, E\otimes A)$, 
we have $m_h \motimes \mathrm{id}(a) = 0$. 

Let $h_i \in C_c(G)$ be a net of compactly supported continuous positive definite functions which converges to 1 uniformly on compact subsets of $G$,  
and satisfies $\sup_{\gamma \in G} |h_i(\gamma)| \leq 1$ for every $i$. 
Then $m_{h_i} \motimes \mathrm{id}$ converges to $\mathrm{id}_{C^*_r(E)\motimes A}$ in the point-norm topology. 
Hence, for $a \in \ker Q$, we have 
\begin{equation*}
a = \lim_i m_{h_i} \motimes \mathrm{id}(a) = 0
\end{equation*}
by the above argument. This proves that $Q$ is injective. 
\qed

\begin{rmk}
Even when considering a non-amenable groupoid $G$, the $C^*$-algebra $C^*_r(E)$ often happens to be nuclear for some Fell bundle $E$ over $G$. 
For example, if a (non-amenable) discrete group $\Gamma$ acts amenably on a unital nuclear $C^*$-algebra $A$, 
then the reduced crossed product $A \rtimes_r \Gamma$ is nuclear (see \cite[Section 4]{BO}). 

\end{rmk}

\section{Examples and Applications}

\subsection{Crossed Products}
\begin{dfn}(cf.~\cite{K}) 
Let $q:A\rightarrow G^{(0)}$ be a $C^*$-algebraic bundle over $G^{(0)}$. 
Define $G*A = \{ (\gamma, a)\ |\ s(\gamma)=q(a) \}$. 
An {\it action} of $G$ on $A$ is a continuous map $\alpha:G*A \rightarrow A$, $\alpha(\gamma,a)$ is denoted by $\alpha_{\gamma}(a)$, 
satisfying the following conditions: 
\begin{itemize}
\setlength{\itemsep}{-11pt}
\item[{\rm (i)}] $q(\alpha_{\gamma}(a)) = r(\gamma)$. \\
\item[{\rm (ii)}] $\alpha_{\gamma}:A_{s(\gamma)}\rightarrow A_{r(\gamma)}$ is an isomorphism. \\ 
\item[{\rm (iii)}] $\alpha_x(a)=a$ for $x \in G^{(0)}$ and $a \in A_x$. \\
\item[{\rm (iv)}] $\alpha_{\gamma_1 \gamma_2}(a) = \alpha_{\gamma_1}(\alpha_{\gamma_2}(a))$. 
\end{itemize}
\end{dfn}

If there is an action $\alpha$ of a groupoid $G$ on a $C^*$-algebraic bundle $A$ over $G^{(0)}$, 
we can define a Fell bundle structure on $G*A$. 
The topology of $G*A$ is the relative topology of the usual product topology, 
and the projection $p:G*A\rightarrow G$ is the projection onto the first coordinate. 
The norm on $G*A$ is defined by $\|(\gamma,a)\| = \|a\|$. 
The multiplication and the involution is defined by 
\begin{eqnarray*}
&&(\gamma_1, a_1)(\gamma_2, a_2) = (\gamma_1 \gamma_2, \alpha_{\gamma_2^{-1}}(a_1) a_2),\\
&&(\gamma,a)^*=(\gamma^{-1},\alpha_{\gamma}(a^*)). 
\end{eqnarray*}
$G*A$ is called a {\it semidirect-product bundle}. 

A semidirect-product bundle is a kind of generalization of usual crossed products by groups to groupoids. 
If $G=\Gamma$ is a discrete group, then $\Gamma^{(0)}=\{e\}$ and a $C^*$-algebraic bundle over $\Gamma^{(0)}$ is just a $C^*$-algebra. 
In this case, $C_c(\Gamma*A)$ is identified with finitely-supported $A$-valued functions on $\Gamma$, 
and the multiplication and the involution on $C_c(\Gamma*A)$ coincide with those of crossed products. 
Therefore, we have $C_r(\Gamma*A)\cong A\rtimes_r \Gamma$ since both sides are completions of $C_c(\Gamma*A)$ 
characterized by the canonical faithful expectation onto $A$. 
In the case of crossed products by groups, Theorem \ref{nuclear} reduces to (iii) in Theorem \ref{classical}. 

The semidirect product bundle is certainly a proper generalization of crossed products, 
but it may seem less natural that actions on bundles are used instead of actions of $C^*$-algebras. 
Indeed, we can define crossed products by groupoids from actions on cartain class of $C^*$-algebras---
{\it $C_0(X)$-algebras} (where $X$ is the unit space of the groupoid). 
Actually, we can describe actions on $C_0(X)$-algebras in the form of bundles, 
but we need {\it upper-semicontinuous Fell bundle} to handle the general case, 
which requires only the upper-semicontinuity instead of the continuity of norms in the axioms of Banach bundles. 
See \cite{MW2}, \cite{MW} for general crossed products and the treatment of upper-semicontinuous bundles. 

\subsection{Line Bundles}
A Fell bundle $E$ over a groupoid is called a Fell line bundle if each fiber of $E$ is one-dimensional. 
If $X$ is a locally compact space, then there is only one Fell line bundle over $X$, i.e., the trivial one. 
For a groupoid $G$, there is a canonical one-to-one correspondence between Fell line bundles over $G$ 
and {\it twists} over $G$---an extended notion of the second cohomology of groupoids. 
See \cite{K2},\cite{R} for twists and \cite{R2} for cohomologies. 

Fell line bundles are related to the theory of Cartan subalgebras for $C^*$-algebras. 
In the well-known theorem of Feldmam-Moore, a von Neumann algebra which have a Cartan subalgebra is described 
by an equivalence relation and a 2-cocycle (\cite{FM1},\cite{FM2}). 
In the case of $C^*$-algebras, equivalence relations are replaced by {\it topologically principal} \etale groupoids and 
2-cocyles are replaced by twists. 

\begin{dfn}(Renault, \cite{R}) 
An abelian $C^*$-subalgebra $B$ of a $C^*$-algebra $A$ is said to be a Cartan subalgebra if we have the following conditions:
\begin{itemize}
\setlength{\itemsep}{-11pt}
\item[{\rm (i)}] The subalgebra $B$ contains an approximate unit of $A$. \\
\item[{\rm (ii)}]The subalgebra $B$ is a maximal abelian subalgebra of $A$. \\ 
\item[{\rm (iii)}] The normalizer $N(B)$ of $B$ generates $A$. \\
\item[{\rm (iv)}] There exists a faithful conditional expectation of $A$ onto $B$. 
\end{itemize}
Here, $N(B) = \{a \in A\ |\ aBa^*, a^*Ba \subset B \}$. 
\end{dfn}

Renault (\cite{R}) defined the notion of Cartan subalgebras for $C^*$-algebras and proved a theorem analogous to the theorem of Feldman-Moore. 
\begin{thm}{\rm (Renault, \cite{R})} 
If $G$ is a locally compact Hausdorff \etale topologically principal second countable groupoid and $E$ is a Fell line bundle over $G$, 
then $C_0(E^{(0)}) = C_0(G^{(0)})$ is a Cartan subalgebra of $C^*_r(E)$. 
Conversely, if $A$ is a separable $C^*$-algebra and $B$ is a Cartan subalgebra of $A$, 
there exists a locally compact Hausdorff \etale topologically principal second countable groupoid $G$, a Fell line bundle $E$ over $G$, and 
an isomorphism of $C^*_r(E)$ onto $A$ which carries $C_0(G^{(0)})$ onto $B$. 
The groupoid $G$ and the Fell line bundle $E$ are unique up to isomorphism. 
\end{thm}

In the above theorem, Fell line bundles are used to fetch some extra information lost in the construction of groupoids. 
Therefore, for a groupoid $G$, the reduced $C^*$-algebras associated to Fell line bundles over $G$ are expected to share some properties. 
For example, the nuclearity is the one of such properties (We do not know other examples, but there should be a great deal of such properties). 

\begin{thm}
For a locally compact Hausdorff \etale groupoid $G$, the following conditions are equivalent: 
\begin{itemize}
\setlength{\itemsep}{-11pt}
\item[{\rm (i)}] $G$ is amenable. \\
\item[{\rm (ii)}] $C^*_r(E)$ is nuclear for all Fell line bundles over $G$. \\
\item[{\rm (iii)}] $C^*_r(E)$ is nuclear for some Fell line bundle over $G$. \\
\end{itemize}
\end{thm}
\begin{proof}
The condition (i) implies (ii) by Theorem \ref{nuclear}, and obviously (ii) implies (iii). We will prove that (iii) implies (i). 
The proof is based on \cite[Theorem 5.6.18]{BO}. 

Let $E$ be a Fell bundle over $G$ such that $C^*_r(E)$ is nuclear. 
By the definition of amenability, it suffices to show that there exists a net $\{h_i\}$ of compactly supported positive definite functions on $G$ 
with $\sup_{x \in G^{(0)}} |h(x)| \leq 1$ which converges to $1$ uniformly on compact subsets of $G$. 
We will construct $h_i$ of the form $h_i = \zeta_i^**\zeta_i$ with $\|\zeta_i \|_{L^2(G)}\leq 1$ , where
\begin{equation*}
\|\zeta_i \|_{L^2(G)} = \sup_{x \in G^{(0)}} \left( \sum_{\gamma \in G_x} | \zeta_i(\gamma) |^2 \right)^{1/2}. 
\end{equation*}

Fix $\varepsilon >0$ and a compact subset $K$ of $G$ such that $K^{-1}=K$.
Take relatively compact bisections $U_1,\cdots,U_n,V_1,\cdots,V_n$ of $G$ satisfying 
$K \subset U_1 \cup \cdots \cup U_n$, $\overline{U_l} \subset V_l$ and that $E$ is trivial on each $V_l$. 
Note that Banach line bundles are always ``locally trivial" by \cite[Remark 13.19]{FD1}. 
Let $\Phi_l : E|_{V_l} \rightarrow V_l \times \C$ be an isomorphism as Banach bundles over $V_l$. 
Then $\Phi_l$ is an isomorphism as complex vector bundles and isometric on each fiber. 
Let $\tilde{f_l}$ be a continuous function on $G$ such that $\tilde{f_l}$ is identical to $1$ on $\overline{U_l}$, $0\leq \tilde{f_l} \leq 1$ 
and the support of $\tilde{f_l}$ is contained in $V_l$. 
Define a continuous section $f_l$ of $E$ by 
\begin{eqnarray*}
f_l(\gamma) = \left\{ 
\begin{array}{cc} \Phi_l^{-1}(\gamma, \tilde{f_l}(\gamma)) & (\gamma \in V_l), \\
0_{\gamma} & (\gamma \not\in V_l), \end{array} \right. 
\end{eqnarray*}
where $0_{\gamma}$ is the zero element of $E_\gamma$. 

Since $C^*_r(E)$ is nuclear by the assumption, there exist c.c.p.~maps $\psi:C^*_r(E)\rightarrow M_n(\C)$ and $\varphi:M_n(\C)\rightarrow C^*_r(E)$ satisfying 
\begin{eqnarray*}
&&\|\varphi \circ \psi (f_l) - f_l\| \leq \varepsilon, \\
&&\|\varphi \circ \psi (f_l^**f_l) - f_l^**f_l\| \leq \varepsilon
\end{eqnarray*}
for $l=1,\cdots,n$. Note that $f_l^**f_l \in C_c(E^{(0)})$ and $\| f_l^**f_l \| = \| f_l \|_{\infty}^{1/2} = 1$. 

Let $\{e_{ij}\}_{i,j}$ be a matrix unit of $M_n(\C)$ and let $b = [\varphi(e_{ij})]_{i,j}^{1/2} \in M_n(C^*_r(E))$. 
Note that $[\varphi(e_{ij})]_{ij} \geq 0$ since $[e_{ij}]_{ij} \geq 0$ in $M_n(M_n(\C))$. 
Put $\eta_{\varphi} = \sum_{j,k} \xi_j \otimes \xi_k \otimes b_{kj} \in \ell^2_n \otimes \ell^2_n \otimes C^*_r(E)$, where 
$\{\xi_j\}_{j=1}^n$ is an orthonormal basis of the $n$-dimensional Hilbert space $\ell^2_n$ (the inner product is assumed to be linear for the second-variable). 
We regard $\ell^2_n \otimes \ell^2_n \otimes C^*_r(E)$ as a right Hilbert $C^*_r(E)$-module in the natural way. 
Then we have $\langle \eta_{\varphi}, (a\otimes 1\otimes 1)\eta_{\varphi}\rangle = \varphi(a)$ for $a \in M_n(\C)$ since 
\begin{eqnarray*}
\langle \eta_{\varphi}, (a\otimes 1\otimes 1)\eta_{\varphi}\rangle 
= \sum_{j,k,l} \langle \xi_j, a\xi_l \rangle b_{kj}^*b_{kl}
= \sum_{j,l} a_{jl} \varphi(e_{jl}) = \varphi(a). 
\end{eqnarray*}
In particular, $\|\eta_{\varphi}\|^2 = \|\varphi(1)\| \leq 1$. 

Put $a_l = \psi(f_l)$. 
Then we have $|(a_l \otimes 1 \otimes 1)\eta_{\varphi} - \eta_{\varphi}f_l|^2 \leq 3\varepsilon$ since 
\begin{eqnarray*}
&&|(a_l \otimes 1 \otimes 1)\eta_{\varphi} - \eta_{\varphi}f_l|^2 \\
&&= | (a_l \otimes 1 \otimes 1)\eta_{\varphi} |^2 
- \langle \eta_{\varphi}f_l,(a_l \otimes 1 \otimes 1)\eta_{\varphi} \rangle 
- \langle (a_l \otimes 1 \otimes 1)\eta_{\varphi},\eta_{\varphi}f_l \rangle 
+ | \eta_{\varphi}f_l |^2 \\
&&= \varphi(a_l^*a_l) -f_l^*\varphi(a_l) -\varphi(a_l^*)f_l +f_l^*f_l \\
&&= (\varphi(a_l^*a_l) - f_l^*f_l) +f_l^*(f_l - \varphi(a_l)) + (f_l - \varphi(a_l))^*f_l \\
&&\leq 3\varepsilon. 
\end{eqnarray*}
Let us take $\eta'_{\varphi} \in \ell^2_n\odot \ell^2_n\odot C_c(E)$ satisfying 
$\|\eta_{\varphi} - \eta'_{\varphi}\| \leq \varepsilon$ and $\|\eta'_{\varphi}\| \leq 1$. 
We write $\eta'_{\varphi} = \sum_{j,k} \xi_j \otimes \xi_k \otimes \zeta_{kj}$ for some $\zeta_{kj} \in C_c(E)$. 
For $\gamma \in G$, define 
\begin{equation*}
\eta'_{\varphi}(\gamma) = \sum_{j,k} \xi_j \otimes \xi_k \otimes \zeta_{kj}(\gamma) \in \ell^2_n \otimes \ell^2_n \otimes E_{\gamma}, 
\end{equation*}
where $\ell^2_n \otimes \ell^2_n \otimes E_{\gamma}$ is considered as a right Hilbert $E_{s(\gamma)}$-module. 
Since $E_{s(\gamma)}$ is isomorphic to $\C$, this is just an $n^2$-dimensional Hilbert space. 
Put $\zeta(\gamma) = \| \eta'_{\varphi}(\gamma) \|$. 
Then $\zeta$ is a compactly-supported continuous function on $G$ and satisfies $\| \zeta \|_{L^2(G)} \leq 1$ because 
\begin{eqnarray*}
\| \zeta \|_{L^2(G)}^2 &=& \sup_{x \in G^{(0)}} \sum_{\gamma \in G_x} \sum_{k,j} \| \zeta_{kj}(\gamma)^*\zeta_{kj}(\gamma) \| \\
&=& \sup_{x \in G^{(0)}} \left\| \sum_{k,j} \sum_{\gamma \in G_x} \zeta_{kj}(\gamma)^*\zeta_{kj}(\gamma) \right\| \\
&=& \sup_{x \in G^{(0)}} \| \langle \eta'_{\varphi},\eta'_{\varphi} \rangle (x) \| \leq 1. 
\end{eqnarray*}

Hence, all of we have to show is that $\sup_{\gamma \in K} |1-\zeta^**\zeta(\gamma)|$ is small. 
First, we have $\| f_l^**f_l - \varphi(\psi(f_l)^*\psi(f_l)) \| \leq 4\varepsilon$ since
\begin{eqnarray*}
f_l^**f_l &\approx_{2\varepsilon}& \varphi \circ \psi(f_l)^*\varphi \circ \psi(f_l) \\
&\leq& \varphi(\psi(f_l)^*\psi(f_l)) \\
&\leq& \varphi \circ \psi (f_l^**f_l) \approx_{\varepsilon} f_l^**f_l. 
\end{eqnarray*}
Hence we have
\begin{eqnarray*}
f_l^**f_l &\approx_{4\varepsilon}& \varphi(a_l^*a_l) \\
&=& \langle (a_l\otimes 1\otimes 1)\eta_{\varphi},(a_l\otimes 1\otimes 1)\eta_{\varphi}\rangle \\
&\approx_{\sqrt{3\varepsilon}}& \langle (a_l\otimes 1\otimes 1)\eta_{\varphi},\eta_{\varphi}f_l\rangle \\
&\approx_{2\varepsilon}& \langle (a_l\otimes 1\otimes 1)\eta'_{\varphi},\eta'_{\varphi}f_l\rangle .
\end{eqnarray*}

Fix $1\leq l\leq n$ and $\gamma \in U_l$. Put $x=s(\gamma)$. 
Since $|f_l(\gamma)|^2$ is a positive element of norm $1$, we have $f_l^**f_l(x)=|f_l(\gamma)|^2=1$, 
where $1$ is the unit of the $C^*$-algebra $E_x$ (which is isomorphic to $\C$). 
In addition, for $\beta \in G_x$,
\begin{eqnarray*}
\eta'_{\varphi}f_l(\beta) = \sum_{\alpha \in G_x} \eta'_{\varphi}(\beta \alpha^{-1})f_l(\alpha)
= \eta'_{\varphi}(\beta \gamma^{-1})f_l(\gamma). 
\end{eqnarray*}
Hence, $\|\eta'_{\varphi}f_l(\beta)\| \leq \|\eta'_{\varphi}(\beta \gamma^{-1})\|$. 

Finally, we have 
\begin{eqnarray*}
1 &=& \| f_l^**f_l(x) \| \\
&\approx& \| \langle (a_l\otimes 1\otimes 1)\eta'_{\varphi},\eta'_{\varphi}f_l\rangle (x) \| \\
&=& \left\| \sum_{\beta \in G_x} \langle (a_l\otimes 1\otimes 1)\eta'_{\varphi}(\beta),\eta'_{\varphi}f_l(\beta) \rangle \right\| \\
&\leq& \sum_{\beta \in G_x} \|\eta'_{\varphi}(\beta)\| \|\eta'_{\varphi}f_l(\beta)\| \\
&\leq& \sum_{\beta \in G_x} \|\eta'_{\varphi}(\beta)\| \|\eta'_{\varphi}(\beta \gamma^{-1})\| \\
&=& \zeta^**\zeta(\gamma^{-1}) \leq 1
\end{eqnarray*}
and therefore $1 \approx \zeta^**\zeta(\gamma^{-1})$. 
This leads to the conclusion since $K^{-1}=K$ and the error in the above does not depend on $\gamma$. 
\end{proof}
 
\bibliographystyle{amsplain} 

\bibliography{ref.bib} 

\providecommand{\bysame}{\leavevmode\hbox to3em{\hrulefill}\thinspace}
\providecommand{\MR}{\relax\ifhmode\unskip\space\fi MR }
\providecommand{\MRhref}[2]{%
  \href{http://www.ams.org/mathscinet-getitem?mr=#1}{#2}
}
\providecommand{\href}[2]{#2}
\begin{thebibliography}{10}

\bibitem{A}
F.~Abadie-Vicens, \emph{Tensor products of {F}ell bundles over discrete
  groups}, arXiv:funct-an/9712006v1 (1997).

\bibitem{Av}
D.~Avitzour, \emph{Free products of {$C^*$}-algebras}, Trans. Amer. Math. Soc.
  \textbf{271} (1982), no.~2, 423--435.

\bibitem{BO}
N.~P. Brown and N.~Ozawa, \emph{{$C^*$}-algebras and finite-dimensional
  approximations}, Graduate Studies in Mathematics, vol.~88, American
  Mathematical Society, 2008.

\bibitem{FM1}
J.~Feldman and C.~C. Moore, \emph{Ergodic equivalence relations, cohomology,
  and von {N}eumann algebras. {I}}, Trans. Amer. Math. Soc. \textbf{234}
  (1977), no.~2, 289--324.

\bibitem{FM2}
\bysame, \emph{Ergodic equivalence relations, cohomology, and von {N}eumann
  algebras. {II}}, Trans. Amer. Math. Soc. \textbf{234} (1977), no.~2,
  325--359.

\bibitem{FD1}
J.~M.~G. Fell and R.~S. Doran, \emph{Representations of *-algebras, locally
  compact groups, and {B}anach *-algebraic bundles}, vol.~1, Academic Press,
  1988.

\bibitem{FD2}
\bysame, \emph{Representations of *-algebras, locally compact groups, and
  {B}anach *-algebraic bundles}, vol.~2, Academic Press, 1988.

\bibitem{KW}
E.~Kirchberg and S.~Wassermann, \emph{Operations on continuous bundles of
  {$C^*$}-algebras}, Math. Ann. \textbf{303} (1995), no.~4, 677--697.

\bibitem{K2}
A.~Kumjian, \emph{On {$C\sp \ast$}-diagonals}, Canad. J. Math. \textbf{38}
  (1986), no.~4, 969--1008.

\bibitem{K}
\bysame, \emph{Fell bundles over groupoids}, Proc. Amer. Math. Soc.
  \textbf{126} (1998), no.~4, 1115--1125.

\bibitem{MW2}
P.~S. Muhly and D.~P. Williams, \emph{Equivalence and disintegration theorems
  for {F}ell bundles and their {$C\sp *$}-algebras}, Dissertationes Math
  \textbf{456} (2008), 1--57.

\bibitem{MW}
\bysame, \emph{Renault's equivalence theorem for groupoid crossed products},
  NYJM Monographs, vol.~3, State University of New York, University at Albany,
  Albany, NY, 2008.

\bibitem{Q}
J.~C. Quigg, \emph{Discrete {$C^*$}-coactions on {$C^*$}-algebras}, J. Austral.
  Math. Soc. Ser. \textbf{60} (1996), no.~2, 204--221.

\bibitem{RW}
I.~Raeburn and D.~P. Williams, \emph{Morita equivalence and continuous-trace
  {$C^*$}-algebras}, Mathematical Surveys and Monographs, vol.~60, American
  Mathematical Society, 1998.

\bibitem{R2}
J.~Renault, \emph{A groupoid approach to {$C\sp \ast$}-algebras}, Lecture Notes
  in Mathematics, vol. 793, Springer, 1980.

\bibitem{R}
\bysame, \emph{Cartan subalgebras in {$C\sp *$}-algebras}, Irish Math. Soc.
  Bull. (2008), no.~61, 29--63.

\end{thebibliography}

\end{document}